\documentclass[11pt]{article} 
\usepackage{a4,amssymb, amsmath, amsthm}
\usepackage{color}

\newtheorem{thm}{Theorem}[section]
\newtheorem{cor}[thm]{Corollary}

\newtheorem{lem}[thm]{Lemma}

\theoremstyle{remark}
\newtheorem{rem}[thm]{Remark}
\newtheorem{ex}[thm]{Example}

\newcommand{\R}{\mathbb{R}}

\usepackage{graphicx}
\usepackage[top=1in, bottom=1in, left=1in, right=1in]{geometry}

\begin{document}

\title{Asymptotics for optimal design problems for the Schr\"odinger equation with a potential \\\vskip 0.8cm}
\author{Alden Waters\thanks{Department of Mathematics, University of Groningen} \hspace{0.05cm} and Ekaterina Merkurjev\thanks{Department of Mathematics and CMSE, Michigan State University}}

\date{}

\maketitle \vskip 0.5cm
\begin{abstract}
\noindent 

We study the problem of optimal observability and prove time asymptotic observability estimates for the Schr\"odinger equation with a potential in $L^{\infty}(\Omega)$, with $\Omega\subset \mathbb{R}^d$, using spectral theory. An elegant way to model the problem using a time asymptotic observability constant is presented. For certain small potentials, we demonstrate the existence of a nonzero asymptotic observability constant under given conditions and describe its explicit properties and optimal values. Moreover, we give a precise description of numerical models to analyze the properties of important examples of potentials wells, including that of the modified harmonic oscillator. 
 
\end{abstract}



\vskip 1.0cm

\section{Introduction}\label{intro}
Let $\Omega \subset \mathbb{R}^d$, be a bounded domain with boundary $\partial\Omega$. Let $T>0$, and $\omega$ be a measureable subset of $\Omega$. We consider the Schr\"odinger equation with Dirichlet boundary conditions:
\begin{align}\label{nonlinear}
&i\partial_tu=\Delta u-V(x)u\\& \nonumber
u(0,x)=u_0(x)\\& \nonumber
u(t,x)|_{x\in\partial\Omega}=0 \nonumber.
\end{align}
Here, $u:\mathbb{R}\times \Omega\mapsto \mathbb{C}$, $V(x)\in L^{\infty}(\Omega)$ and $u_0(x)=u(0,x) \in H^1_0(\Omega)\cap H^2(\Omega)$. In some instances, we will require higher regularity, but this will be specified when necessary. 

Let $-\Delta$ indicate the Laplacian on the space $C_0^{\infty}(\Omega)\subset L^2(\Omega)$. This operator is a symmetric operator acting on $L^2(\Omega)$ associated with the quadratic form
\begin{align*}
&Q_0: H_0^1(\Omega)\rightarrow [0,\infty)\\&
Q_0(f)=\int\limits_{\Omega}\nabla f(x)\cdot \overline{\nabla f(x)}\,dx.
\end{align*}
In particular we recall that the quadratic form is closable with respect to the norm
\begin{align*}
Q_D(f)=\left(Q_0(f)+||f||_{L^2(\Omega)}\right)^{1/2}.
\end{align*}
The domain of the closure $Q_D(f)$ is the Sobolev space $H_0^1(\Omega)$. We can define the Dirichlet Laplacian $-\Delta_D$ via this extension procedure and moreover, 
\begin{align}
\mathrm{Dom}((-\Delta_D)^{1/2})=\mathrm{Dom}(Q)=H_0^1(\Omega).
\end{align} 
If $\Omega$ is a bounded domain with boundary $\partial\Omega$ of class $C^2$ then 
\begin{align}
\mathrm{Dom}((-\Delta_D))=H_0^1(\Omega)\cap H^2(\Omega).
\end{align}

All of the functions of this operator are interpreted via the Hilbert space functional calculus. In particular, $\exp(it\Delta_D)$ is unitary, and we exploit this property to build our parametrices. The representation of the solutions presented here in the case of an added potential $V$ is new and relies on applications of advanced spectral theory. 

If we consider the Schr\"odinger equation on a bounded domain $\Omega$ of $\R^d$ with Dirichlet boundary conditions, then observing the restriction of the solutions to a measurable subset $\omega$ of $\Omega$ during a time interval $[0,T]$ with $T>0$, is known as \textit{observability}. The equation (1) is {\it observable} on $\omega$ in time $T$ if there exists $C>0$ such that
\begin{equation}\label{myeq}
C ||\partial_tu(0,x)||_{L^2(\Omega)}^2\leq \int_0^T \int_\omega |\partial_tu(t,x)|^2\ dt\ dx.
\end{equation}
In previous literature, the above inequality is called the observability inequality when $V=0$. 

It is well known that if the pair $(\omega,T)$ satisfies the observability inequality \eqref{myeq}, then the energy of the solutions can be estimated in terms of the energy which is localized in $\omega\times (0,T)$. The search is then for the conditions on $\omega$ for which one can find the largest possible non-negative constant for which the inequality \eqref{myeq} holds. 

We denote the {\it observability constant} by $C^V_T(\chi_\omega)$ to be the largest constant such that \eqref{myeq} holds. The constant can also be formulated as:
\begin{align}
C_T^V(\chi_{\omega})=\left\{ \inf \frac{\int\limits_0^T\int\limits_{\omega}|\partial_tu(t,x)|^2\,dx\,dt}{||\partial_tu(0,x)||_{L^2(\Omega)}^2}| \,\, \, u(0,x)\in H_0^1\cap H^2(\Omega) \right\}.
\label{observability_constant}
\end{align}

The study of the observability constant $C^V_T(\chi_\omega)$ is important, since it gives an account for the well-posedness of the inverse problem of reconstructing $u$ from measurements over $[0,T]\times \omega$. In addition, we denote $C_T^0(\chi_{\omega})$ as the constant associated with the Schr\"odinger equation without a potential. The main novelty of the paper is the analysis in the case of an added potential $V$. 

We now connect the theory to a possible real-life application. Assume that $\Omega\subset \mathbb{R}^d$ is a cavity in which signals are propagating according to (1). To measure the propagating signals, one is allowed to place a few sensors into the cavity. We now assume that, in addition to the placement of the sensors, we are allowed to choose their shape. Therefore, the problem is now of determining the best possible location and shape of the sensors, which will obtain the best observation. Of course, the best choice is to observe the solutions over the whole domain $\Omega$. However, in practice, the domain scanned by the sensors is usually limited, for reasons such as the cost of such an operation. To make this limitation more mathematically precise, we consider measurable subsets of fixed size, i.e. subsets $\omega$ of $\Omega$ such that $|\omega|=L|\Omega|$, where $L \in (0,1)$. The subset $\omega$ represents the sensors in $\Omega$, and they are able to measure restrictions of the solutions of (1) to $\omega$.

    Therefore, one and the most obvious way to model the problem of {\it best observability}, is that of finding the optimal set which maximizes the functional $\chi_\omega \rightarrow C^V_T(\chi_\omega)$ over the set
     \begin{align*}
&\mathcal{M}_{L}=\{\omega \subset \Omega\ |  \ 
 \omega \,\, \textrm{is measurable and of Lebesgue measure}\,\, |\omega|=L|\Omega| \}\ .
\end{align*}
However, we show that this problem is not only inherently difficult to solve, but is not so relevant in practice. Thus, we consider several modifications and simplifications of the model, to be described in the next section.

Optimal observation problems are found in numerous engineering applications, thus providing the motivation for our study. Examples include acoustics, piezoelectric actuators, vibration control in mechanical structures, damage detectors, and chemical reactions \cite{kumar,morris,sigmund,uc,wal}. The goal is to optimize the type and place of the sensors in order to improve the estimation of the overall behavior of the state of the system. 

The main contributions of the paper are the following:

\begin{enumerate}

\item We present an elegant way to model the problem of best observability using the time asymptotic observability constant $C_{\infty}^V(\chi_{\omega})$. We analyze the largest possible $C_{\infty}^V(\chi_{\omega})$, over all $\omega \in \mathcal{M}_{L}$, and we develop conditions analogous to the Quantum Unique Ergodicity conditions in \cite{ptz} for this constant to hold.

\item We demonstrate the conditions on the existence of a positive asymptotic observability constant  $C_{\infty}^V(\chi_{\omega})$ for an arbitrary subset $\omega$ of $\Omega$ and $T>0$, under certain requirements on the potential. Our results are supported by numerical experiments. 
\end{enumerate}

\begin{rem}
The paper \cite{ptz} considers a variety of boundary conditions, but we focus on how to treat the problem with a potential, so we simply impose Dirichlet boundary conditions. Different boundary conditions will be the subject of future study. One could examine the problem on a compact Riemannian manifold $\mathcal{M}$, such that $(\mathcal{M},g)$ has a boundary, and use the Laplace Beltrami-operator $\Delta_g$, and many of the same results would still hold. However, we let $\Omega$ be a subdomain of $\mathbb{R}^d$ for simplicity. 
\end{rem}

\section{Statement of the Main Theorems}

Consider the eigenvalues $(\lambda_j)_{j\in\mathbb{N}}$ and the corresponding eigenfunctions $\phi_j(x)$  for $-\Delta+V(x)$ on $\Omega$. Let $(\lambda_{j0})_{j\in\mathbb{N}}$ and $\phi_{j0}(x)$ denote the eigenvalues and corresponding eigenfunctions of $-\Delta$ on $\Omega$. For the rest of this article we drop the subscript $D$ for the Dirichlet Laplacian. 

We assume the $\phi_j(x)'s$ are orthonormal, and give references to classical spectral theory results which show that they can be used as a basis for $H_0^1(\Omega)$. The solution of (1) can then be represented as
\begin{align}
u(t,x)=\sum\limits_{j=1}^{\infty}c_j\exp(i\lambda_jt)\phi_j(x),
\label{u_expansion}
\end{align}
where $u(0,x)\in H_0^1(\Omega)$ is the initial data to the solution $u(t,x)\in C^0((0,T); H^2(\Omega))$. The sequence
$(c_j)_{j\in\mathbb{N}^*} \in \ell^2(\mathbb{C})$
is determined in terms of $u(0,x)$ as 
\begin{align*}
c_j=\int\limits_{\Omega}u(0,x)\overline{\phi_j(x)}\,dx. 
\end{align*} 
Moreover, 
\begin{align*}
||\partial_tu(0,x)||_{L^2(\Omega)}^2=\sum\limits_{j=1}^{\infty}\lambda_j^2|c_j|^2.
\end{align*}

If 
\begin{align}
G_T^V(\chi_{\omega})= \int\limits_0^T\int\limits_{\omega}|\partial_tu(t,x)|^2\,dx\,dt,
\end{align}
then plugging in expansion \eqref{u_expansion} yields
\begin{align}
&G_T^V(\chi_{\omega})=\\& \nonumber \int\limits_0^T\int\limits_{\omega}|\sum\limits_{j=1}^{\infty}\lambda_jc_j\exp(i\lambda_jt)\phi_j(x)|^2\,dx\,dt=\sum\limits_{j,k=1}^{\infty}\lambda_j\lambda_kc_j\overline{c}_k\alpha_{jk}\int\limits_{\omega}\phi_j(x)\overline{\phi_k(x)}\,dx
\end{align}
with 
\begin{align}
\alpha_{jk}=\int\limits_0^T\exp(i(\lambda_j-\lambda_k)t)\,dt=\frac{2}{\lambda_j-\lambda_k}[\exp(i(\lambda_j-\lambda_k)T-1],
\label{alpha}
\end{align}
whenever $j\neq k$ and $\alpha_{jj}=T$ whenever $j=k$. 

We notice that the determination of the observability constant is now a difficult spectral problem involving many inner products of eigenfunctions over the set $\omega$. Moreover, it is limited in practice, since the observability constant defined by \eqref{observability_constant} describes the worst possible case, which may not occur often in applications. In order to examine the problem further, one can consider the following simplifications: 

\begin{enumerate}
\item One can examine the problem of maximizing $G^V_T(\chi_{\omega})$ over all possible measurable subsets $\omega\in \mathcal{M}_L$, given fixed initial data. In this case, if the optimal set exists, it depends on the initial data that is considered. This problem is still challenging, and also not relevant enough in practice, since initial data is not expected to be fixed, but uniform in nature. Therefore, we focus on the following second simplification, where all initial conditions are taken into account.

\item One can instead consider a time asymptotic observability constant $C_{\infty}^V(\chi_{\omega})$, as in \cite{ptz}. The constant is defined as
\begin{align}
C_{\infty}^V(\chi_{\omega})=\left\{ \inf \lim_{T \rightarrow \infty} \frac{1}{T}\frac{\int\limits_0^T\int\limits_{\omega}|\partial_t u(t,x)|^2\,dx\,dt}{||\partial_tu(0,x)||_{L^2(\Omega)}^2} | u(0,x)\in H_0^1\cap H^2(\Omega) \right\}.
\end{align}
This constant is the non-negative constant for which the time asymptotic observability inequality
\begin{align}
C_{\infty}^V(\chi_{\omega})||\partial_tu(0,x)||_{L^2(\Omega)}^2\leq \lim_{T\rightarrow \infty}\frac{1}{T}\int\limits_0^T\int\limits_{\omega}|\partial_tu(t,x)|^2\,dx\,dt
\end{align}
holds for every $u(0,x)\in H_0^1(\Omega)\cap H^2(\Omega)$. This is where we use the additional assumption $u\in H^2(\Omega)$ so the constant is well-defined. If $\overline{\Omega}$ is of class $C^2$ then this is the entire domain anyway. Shortly, we will show that the time asymptotic observability constant is equal to the {\it randomized observability constant}
\begin{align}
J^V(\chi_{\omega})=\inf\limits_{j\in\mathbb{N}}\int\limits_{\omega}\phi_j^2(x)\,dx.
\end{align}

Note that, from the definition of the observability constant, one obtains the following inequality:
\begin{equation}
    \lim\sup\limits_{T \rightarrow \infty}  \frac{C_{T}^V(\chi_{\omega})}{T} \leq C_{\infty}^V(\chi_{\omega}).
\end{equation}
\end{enumerate}

The randomized observability constant can be derived in the following way. We introduce a field of i.i.d random variables $\{\beta_j\}_{j\in\mathbb{N}}\in \{0,1\}$ which we use to multiply the values of the initial data. Then
\begin{align}
\label{value}
 \inf \lim_{T \rightarrow \infty} \frac{1}{T} \hspace{0.1cm}\mathbb{E}(G_T^V(\chi_{\omega}))=  \inf \lim_{T \rightarrow \infty} \frac{1}{T}  \hspace{0.1cm}\mathbb{E}(\int\limits_0^T\int\limits_\omega|\sum\limits_{j\in\mathbb{N^*}}(-\Delta+V)(c_j\beta_j\phi_j(x))|^2\,dx dt)= \\
 \inf \lim_{T \rightarrow \infty} \frac{1}{T}   \hspace{0.1cm} \mathbb{E}(\int\limits_0^T\int\limits_\omega|\sum\limits_{j\in\mathbb{N^*}} c_j\lambda_j\beta_j\phi_j(x)|^2\,dx dt)
\end{align}
Note that \eqref{value}, under the conditions $$\{\beta_j\}_{j\in\mathbb{N}}\in \{0,1\},\quad \sum\limits_{j}|c_j\lambda_j|^2=1,$$
 is exactly $J^V(\chi_{\omega})=\inf\limits_{j\in\mathbb{N}}\int\limits_{\omega}\phi_j^2(x)\,dx$.

We now state the main theorems of the paper. All of the theorems in this section are formulated for the Schr\"odinger equation \textit{with} a potential, which is the main novelty. The first two results concern an expression for time asymptotic observability constant $C_{\infty}^V(\chi_{\omega})$:
\begin{thm}\label{main1}[Analogue to Theorem 2.6 \cite{ptz}]
For every measureable subset $\omega$ of $\Omega$,
\begin{align}
C_{\infty}^V(\chi_{\omega})=\left\{ \inf \frac{\int\limits_{\omega}\sum_{\lambda\in U}|\sum\limits_{k\in I(\lambda)}c_k\phi_k(x)|^2\,dx}{\sum\limits_{k=1}^{\infty}|c_k|^2}| (c_j)_{j\in\mathbb{N}^*}\in \ell^2(\mathbb{C})\setminus \{0\} \right\},
\end{align}
where $U$ is the set of all distinct eigenvalues $\lambda_k$ and $I(\lambda)=\{j\in\mathbb{N}^*| \lambda_j=\lambda\}$.
\end{thm}
The proof of the theorem is in Section 4. If we set
\begin{align}
J^V(\chi_{\omega})=\inf_{j\in\mathbb{N}}\int\limits_{\omega}\phi_j^2(x)\,dx,
\end{align}
then similarly to Theorem 1 in \cite{ptz}, we have: 
\begin{cor}[Analogous to Corollary 2.7 in \cite{ptz}]\label{cmain1}
The inequality $C_{\infty}^V(\chi_{\omega})\leq J^V(\chi_{\omega})$ is true for every measurable subset $\omega$ of $\Omega$. If the domain $\Omega$ is such that every eigenvalue of $-\Delta+V$ is simple, then
\begin{align}
C_{\infty}^V(\chi_{\omega})=\inf_{j\in\mathbb{N}}\int\limits_{\omega}\phi_j^2(x)\,dx=J^V(\chi_{\omega}) 
\end{align}
for every measurable subset $\omega$ of $\Omega$. This shows that the off-diagonal terms in the eigenfunction expansion contribute less in the infinite time asymptotic regime. \end{cor}

The more difficult problem is using known results from perturbation theory to find a non-zero observability constant. We show that if $V(x)=\varepsilon V_0(x)$ for some $\varepsilon\in (0,1)$, then under certain conditions on $\varepsilon$, we can find a positive time asymptotic observability constant for the Schr\"odinger equation whenever the corresponding operator without the potential ($V(x)\equiv 0$) has one. 

For Theorem \ref{perturbationtheorem} it is assumed that the potential has regularity $V\in L^{\infty}(\Omega)$.  Let $C(V_0,\Omega)$ be a constant which depends uniformly on the diameter of $\Omega$ and the $L^{\infty}(\Omega)$ norm of $V_0$. This constant will be derived and given explicitly during the course of the proof. We prove:

\begin{thm}\label{perturbationtheorem}
We assume that $-\Delta$ and $-\Delta+V(x)$ on $\Omega$ with Dirichlet boundary conditions both have \emph{simple} spectra, for all $\varepsilon\in (0,\varepsilon_0)$ with fixed $\varepsilon_0$ sufficiently small. When $V(x)=\varepsilon V_0(x)$, $\mathrm{supp}V_0\subset\omega$ and $\varepsilon<\frac{1}{C(V_0,\Omega)}$, the constant $J^V(\chi_{\omega})$ is such that $J^V(\chi_{\omega})>0$ if and only if $J^0(\chi_{\omega})>0$ for the Schr\"odinger equation with $V(x)\equiv 0$.
\end{thm}

The proof of Theorem \ref{perturbationtheorem} is in Section \ref{peturbmsec}. Moreover, in Section \ref{peturbmsec}, we discuss why the assumption the spectra is simple is spectrally sharp, as there are counter-examples to the statement of Theorem \ref{perturbationtheorem}  for non-simple spectra given as a result of \cite{MR,MR2, Macia}. There is also an appendix on convergence of numerical algorithms using these functionals. 

We also consider a relaxation of the problem. In particular, let $\overline{\mathcal{M}}_L$ be the convex closure of the set $\mathcal{M}_L$ in the $L^{\infty}$ weak star topology:
\begin{align}
\overline{\mathcal{M}}_L=\{ a\in L^{\infty}(\Omega,[0,1])| \int\limits_{\Omega}a(x)\,dx=L|\Omega|\}.
\end{align}

We set 
\begin{align}
\inf\limits_{j\in\mathbb{N}*}\int\limits_{\Omega}a(x)\phi_j^2(x)\,dx=J_{\varepsilon}(a)
\end{align}
and also 
\begin{align}
\inf\limits_{j\in\mathbb{N}*}\int\limits_{\Omega}a(x)\phi_{j0}^2(x)\,dx=J(a).
\end{align}
We then have, \begin{thm}\label{ngpc}
Let $V=\varepsilon V_0(x)$ with $V_0(x)\in L^{\infty}(\Omega)$, with no assumption on the support of $V(x)$ in $\Omega$. We assume $-\Delta$ and $-\Delta+V(x)$ with Dirichlet boundary conditions both have simple spectrum, all $\varepsilon\in (0,\varepsilon_0)$ with fixed $\varepsilon_0$ sufficiently small. It follows that for any $a\in \overline{\mathcal{M}}_L$
\begin{align}\label{eqnnV} 
|J_{\varepsilon}(a)-J(a)|\leq C_1(V_0,\Omega)\varepsilon^2
\end{align}
with $C_1(V_0,\Omega)$ a constant depending only on $L^{\infty}(\Omega)$ norm of $V_0$ and $|\Omega|$. As a consequence, we can conclude
\begin{align}
\left|\max_{a\in \overline{\mathcal{M}}_L}J_{\varepsilon}(a)-\max_{a\in \overline{\mathcal{M}}_L}J(a)\right|\leq C_1(V_0,\Omega)\varepsilon^2
\end{align}
\end{thm} 
The proof of Theorem \ref{ngpc} is in Section \ref{peturbmsec}. Notice that we cannot show $J_{\varepsilon}(a)=0$ iff $J(a)$ zero because we do not have such fine control over the $\mathcal{O}(\varepsilon^2)$ terms unless $a(x)=\chi_\omega(x)$. 

In the last section, using existing software, we show explicit computations and an explicit representation of the observability constant for a variety of potentials including a damped harmonic oscillator.  While the problem for the non-linear Schr\"odinger equation has been investigated from the control theory standpoint \cite{camille, camille2, bingyu, bingyu2}, to the best of the authors knowledge, the problem of observability for potentials has not been addressed in an explicit way using eigenfunctions and numerical methods. Observability for the linear Schr\"odinger equation was examined in \cite{opt}. Our analysis extends their results in the linear case. 

\subsection{Comparison with Previous Literature}
Let $\omega\subset \Omega$ be any nonempty open set and $T>0$ then there exists a constant $K_T(\chi_\omega)$ such that for any $u_0\in H_0^1(\Omega)$ we have 
\begin{align}\label{K}
||u_0||_{L^2(\Omega)}^2\leq K_T(\chi_\omega)\int\limits_0^T||\exp(it\Delta)u_0||^2_{L^2(\omega)}\,dt
\end{align}
or a constant $B_T(\chi_\omega)$ such that for any $u_0\in H_0^1(\Omega)\cap H^2(\Omega)$
\begin{align}\label{B}
||\Delta u_0||_{L^2(\Omega)}^2\leq B_T(\chi_\omega)\int\limits_0^T||\Delta(\exp(it\Delta)u_0)||^2_{L^2(\omega)}\,dt
\end{align}
depending on the domain of the operator.  

In general the work of Lebeau \cite{L} showed that control (the dual statement to the existence of positive constants $B_T(\chi_\omega)$ or $K_T(\chi_\omega)$) for the Schr\"odinger equation with or without the potential holds under the Geometric Control Condition (GCC):

$\bullet$ There exists $L=L(\Omega,\omega)>0$ such that every geodesic of length $L$ on $\Omega$ intersects $\omega$. \\

Therefore if we let $K_T^V(\chi_{\omega})$ denote the constant with the potential, then $K_T^V(\chi_{\omega})>0$ as soon as the GCC is satisfied. The GCC is also necessary in the case of a smooth potential when the geodesic flow is periodic \cite{Macia}. For the flat torus, Jaffard \cite{Jaffard} and Haraux \cite{Haraux} in 2d and Komoronik \cite{Komoronik} in higher dimensions have shown that this not necessary-- observability holds for any open set $\omega$. Their work was extended to operators with smooth potentials in \cite{BZ1,BZ2}, and also for higher dimensions and time-dependent potentials in \cite{AM}, and for irrational torii and general Schr\"ordinger operators in \cite{AFM}. One can see \cite{Jin} for a literature review and extension to hyperbolic manifolds. 

We look at the constant given by \eqref{observability_constant} which we are examining is a different observability constant when the potential is present and this is distinct from that examined in previous literature. However it is closely related to context analyzed in \cite{AML} which is also done for time dependent potentials on the flat disk and other works. Therefore, the main goal here is to identify in which sense the randomised observability constant with the potential and that without, are \emph{close}. 

When there is no potential, our formulation of the observability constant coincides with the definition \eqref{B}. Indeed, for our formulation, one can re-write \eqref{observability_constant} as 
\begin{align}
C_T^V(\chi_\omega)=\left\{ \inf \frac{\int\limits_0^T\int\limits_{\omega}|(-\Delta+V)\exp(it(-\Delta+V))u_0|^2\,dx\,dt}{||(-\Delta+V)u_0||_{L^2(\Omega)}^2}| \,\, \, u_0\in H_0^1\cap H^2(\Omega) \right\}
\end{align}
The positivity of this constant is \emph{not} directly equivalent to the other two when $V$ is nonzero, as $V$ does not commute with $\exp(it(-\Delta+V))$. The only time the existence of the constants $K_T^V(\chi_\omega)$ and $B_T^V(\chi_\omega)$ could imply the positivity of $C_T^V(\chi_\omega)$ directly is when the potential is positive. However, in the important aforementioned literature \cite{AM, AFM,AML, MR2, Macia,L}, there are several cases in which conditions that insure the positivity of these constants are equivalent-manifolds with periodic geodesic flow, flat tori, and the Euclidean disk. In all these cases, the geometric conditions on the observation set $\omega$ do not depend on the presence of the potential, regardless of whether or not this potential is positive or not-- c.f. the introduction to \cite{MR2}. 

Moreover, Theorem \ref{perturbationtheorem} is proved for the \emph{randomised} observability constant $J^V(\chi_{\omega})$ (otherwise known as the observability constant for eigenfunctions \cite{MR}). It is doubtful such a strong statement is true for the full observability constants $C_T^V(\chi_\omega)$ and $C_T^0(\chi_\omega)$ as the presence of cross terms in \eqref{alpha} are difficult to control when the $\lambda_j's$ are large.  Once again, as in \cite{gw}, the randomised constant can be viewed as the optimistic best case scenario. 

Since Theorem \ref{perturbationtheorem} is only true in the case of sufficiently small and regular potentials of compact support, this shows that even in the case of randomised initial data the observability constants (eigenfunction observability constants) can be very close for strong conditions relating $V$ and $\omega$. It is not that the eigenfunction constants cannot be close for $\mathrm{supp}(V)$ not contained in $\omega$, it is just that the current technique gives much less information about controlling the constants in terms of each other. Hence, Theorem 2.4 has a weaker formulation of the relationship of the relaxed constant $J_{\varepsilon}(a)$ with $\varepsilon$ dependent potential to the original one $J(a)$, and there is no assumption on the support of $V$ with respect to $\omega$. In general, showing observability for randomized initial data (otherwise known as observability of eigenfunctions) is possible under conditions on the observation region $\omega$ which are independent of $V$ for generic potentials c.f. \cite{MR2}. 

The main tools in this article are opposite those of the general tract of semi-classical analysis papers. Previous techniques take advantage of the spectral theorem to turn the high frequency eigenvalues $\lambda_j$ into the semi-classical parameter $h^{-2}$.  Heuristically $-h^2\Delta+h^2V$ as a semi-classical operator has symbol $|\xi|^2+h^2V$, while $-h^2\Delta+\varepsilon h^2V$ has symbol $|\xi|^2+h^2\varepsilon V$ but in the latter case the Hamiltonian ray path $x(t)$ over which solutions are concentrated can be made sufficiently close to that of $|\xi|^2$ if $\varepsilon$ is sufficiently small, as long as $h\leq 1$ which is proved in \cite{JW}, Lemma 8.3. The methodology in \cite{JW} fails here because approximate solutions can only be constructed under a non-trapping condition. 

Because we are exploiting the small parameter $\varepsilon$, we use classical perturbation theory techniques rather than semi-classical analysis. Here we see that classical perturbation theory gives new information in the case when the eigenvalues are simple, which cannot be explained by entirely semi-classical techniques. Moreover the results are applicable to any eigenfunction/eigenvalue pair, not just the high frequency ones. 

However, in this particular case examined in this article, if we rescale so that $h=\lambda^{-\frac{1}{2}}$, then the eigenvalue/eigenvector problem becomes $(-h^2\Delta+h^2V)u=u$, with symbol $|\xi|^2+h^2V(x)$, which in the case of the two dimensional flat disk, and the surface of a sphere, can be solved almost explicitly using semi-classical methods to a high degree of success, c.f. \cite{MR, MR2}, corresponding to \emph{high} frequency eigenvalues in this scenario. In other geometries this is not the case, and these are the settings which we seek to begin to resolve in this article. 

\section{Review of Spectral Theory} 
Suppose $\Omega\subset \mathbb{R}^d$ is a bounded domain in $\mathbb{R}^d$. Then, as in the introduction the Laplace operator $-\Delta$ with Dirichlet boundary conditions can be defined as the self-adjoint operator with the quadratic form $Q_0(f)$
\begin{align}
Q_0(f)=<\nabla f,\overline{\nabla f}>_{L^2(\Omega)}
\end{align}
 with domain $H_0^1(\Omega)$. Because the space $H_0^1(\Omega)$ is compactly embedded in $L^2(\Omega)$ by Rellich's theorem, the spectrum of this operator is purely discrete and has infinity as its only possible accumulation point, c.f., \cite{alex} for a review. Hence, there exists an orthonormal basis $(\phi_{j0})_{j\in\mathbb{N}}$ in $L^2(\Omega)$ consisting of eigenfunctions with eigenvalues $(\lambda_j)$, which we assume to be ordered:
\begin{align}\label{ordering}
&-\Delta \phi_{j0}=\lambda_{j0}\phi_{j0}\\& \nonumber
||\phi_{j0}||_{L^2(\Omega)}=1\\& \nonumber
\phi_{j0}|_{\partial\Omega}=0 \\& \nonumber 
0\leq \lambda_{10}\leq \lambda_{20}\leq . . . \nonumber
\end{align}

Recall that a linear subspace $\mathcal{D}$ of the domain of a closed quadratic form $Q$ is called a core for $Q$ if $Q$ is the closure of its restriction to $\mathcal{D}$. We now recall the following result from \cite{Davies}: 
\begin{thm}[Thm 8.2.1 in \cite{Davies}]\label{selfadjoint}
If $0\leq V\in L^1_{loc}(\Omega)$ and $\Omega\subset\mathbb{R}^d$ is a domain in $\mathbb{R}^d$, then the quadratic form
\begin{align*}
\\& Q_V(f)=Q_0(f)+Q_1(f)=\\& \int\limits_{\Omega}|\nabla f|^2+V|f|^2\,dx,
\end{align*}
which is defined on 
\begin{align}
\mathrm{Dom}(Q_V)=\mathrm{Dom}(Q_0)\cap \mathrm{Dom}(Q_1)
\end{align}
is the form of a non-negative self-adjoint operator $H$. The space $C_c^{\infty}(\Omega)$ is a core for $Q$. 
\end{thm} 

\begin{rem}
We could reduce the assumption on the potential from $L^{\infty}(\Omega)$ to $L^1_{loc}(\Omega)$ using the above theorem in many of the following sections. 
\end{rem}

We also require the following useful result on self-adjoint operators from the same monograph \cite{Davies}:
\begin{thm}[Theorem 8.2.3, Corollary 4.4.3, \cite{Davies}]\label{domain}
If $H$ is defined on $L^2(\Omega)$ by $Hf=-\Delta f+V f$, where $V\in L^{\infty}$, then $H$ is a self-adjoint and bounded below with the same domain as $H_0=-\Delta$. 
\end{thm}

We also have that 
\begin{thm}[Thm 6.3.1 in \cite{Davies}]\label{asymp}
For all bounded domains $\Omega\subset\mathbb{R}^d$, the operator $-\Delta$ has an empty essential spectrum and compact resolvent. The eigenvalues $\{\lambda_n\}_{n=1}^{\infty}$ of $-\Delta$ written in increasing order and repeated according to multiplicity satisfy 
\begin{align}
b_1n^{\frac{2}{d}}\leq \lambda_n\leq b_2n^{\frac{2}{d}}
\end{align}
for some $b_1,b_2>0$ depending only on the geometry of $\Omega$ and $n\geq1$. 
\end{thm}
 As such, $b_1$ and $b_2$ can be made arbitrarily close to one another, if $n$ is large c.f. the proof of Theorem 6.3.1 in \cite{Davies}.  The eigenvalues of $-\Delta$ depend monotonically upon the region $\Omega$ and so can be bounded above and below by the eigenvalues of the cubes which are contained (and, respectively, contain) $\Omega$. It follows from Theorem \ref{domain} that $H_0^1(\Omega)=\mathrm{Dom} Q(f)$. From this fact and the Spectral Theorem, we can conclude from Theorem \ref{selfadjoint}:
\begin{cor}\label{ordering}
For $V\in L^{\infty}(\Omega)$, if 
\begin{align}
\lambda_0\leq . . .\leq \lambda_j \leq \lambda_{j+1} \dots
\end{align}
with $\{\phi_j\}_{j\in\mathbb{N}}$ an orthonormal Hilbert basis of $H_0^1(\Omega)$ consisting of eigenfunctions of the Dirichlet operator $-\Delta+V$ on $\Omega$, which is associated with the  eigenvalues $\{\lambda_j\}_{j\in\mathbb{N}}$, then we can write the propagated solution as 
\begin{align}
u(t,x)= \sum\limits_{j\in\mathbb{N}}c_j\phi_j(t,x)
\end{align}
with 
\begin{align}
c_j=\int\limits_{\Omega}u(0,x)\overline{\phi_j(x)}\,dx. 
\end{align}
\end{cor}
We use the basis properties in Corollary \ref{ordering} in the next section. 

\section{Proof of Theorem \ref{main1}}
The basic idea is to use 
\begin{align}
u(t,x)=\sum\limits_{j=1}^{\infty}c_j\exp(i\lambda_jt)\phi_j(x)
\end{align}
as the decomposition for the solution of \eqref{nonlinear}, where $(\lambda_j,\phi_j(x))_{j\in\mathbb{N^*}}$ are the eigenvalue and eigenfunction pairs for the $-\Delta+V$ operator. One can apply similar steps as \cite{ptz} to prove Theorem \ref{main1}. Using a standard density argument, the approximation which holds over a finite number of modes,
\begin{align}
u(t,x)\approx\sum\limits_{j=1}^{N}c_j\exp(i\lambda_jt)\phi_j(x),
\end{align}
is enough to describe an observability constant which is valid in the large-time regime. Then, we use previously derived facts about perturbation theory to prove the other theorems. 
 
\begin{proof}[Proof of Theorem \ref{main1}]
We start with the case when $-\Delta+V$ has simple eigenvalues. This proof is a simplification of the analogous theorem in \cite{ptz} which is presented for the wave equation and applicable to the Schr\"odinger equation with no potential. Without loss of generality, one can consider initial data such that 
$||\partial_tu(0,x)||_{L^2(\Omega)}=1$. Then, let
\begin{align}
\Sigma_T=\frac{1}{T}\frac{G_T(\chi_{\omega})}{||\partial_tu(0,x)||_{L^2(\Omega)}}=\frac{1}{T}G_T(\chi_{\omega})
\end{align}
and 
\begin{align}
y_j(t,x)=i\lambda_jc_j\exp(i\lambda_jt)\phi_j(x).
\end{align}
Then, $\Sigma_T$ can be expressed as:
\begin{align}
\Sigma_T=\frac{1}{T}\int\limits_0^T\int\limits_{\omega}\left(|\sum\limits_{j=1}^Ny_j(t,x)|^2+|\sum\limits_{k=N+1}^{\infty}y_k(t,x)|^2+2 \left(\sum\limits_{j=1}^Ny_j(t,x)\sum\limits_{k=N+1}^{\infty}\overline{y_k(t,x)}\right)\,dx\,dt \right).
\label{sigma}
\end{align}
Note that 
\begin{equation}
    C_{\infty}^V(\chi_\omega)= \inf \lim_{T \rightarrow \infty} \Sigma_T.
\end{equation}
Now, we use the assumption that the spectrum of $-\Delta+V$ consists of simple eigenvalues to prove the following result:
\begin{lem}
The following equation holds:
\begin{align}
\lim_{T\rightarrow\infty}\frac{1}{T}\int\limits_0^T\int\limits_{\omega}|\sum\limits_{j=1}^Ny_j(t,x)|^2\,dx\,dt=\sum\limits_{j=1}^N\lambda_j^2|c_j|^2\int\limits_{\omega}\phi_j^2(x)\,dx. 
\end{align}
\end{lem}
Because the sum is finite, one can invert the inf and the limit. We have
\begin{align}
&\frac{1}{T}\int\limits_0^T\int\limits_{\omega}|\sum\limits_{j=1}^Ny_j(t,x)|^2\,dx\,dt=\\&
\frac{1}{T}\sum\limits_{j=1}^N\lambda_j^2\alpha_{jj}\int\limits_{\omega}\phi_j^2(x)\,dx+\frac{1}{T}\sum\limits_{j=1}^N\sum\limits_{k=1,k\neq j}^N\lambda_j\lambda_kc_j\overline{c_k}\alpha_{jk}\int\limits_{\omega}\phi_j(x)\overline{\phi_k(x)}\,dx,
\end{align}
where $\alpha_{jk}$ was given previously by \eqref{alpha}.
Formula \eqref{alpha} gives
\begin{align}
\lim\limits_{T\rightarrow \infty}\frac{\alpha_{jj}}{T}=1
\end{align}
for every $j\in\mathbb{N}$. We note that
\begin{align}
|\alpha_{jk}|\leq \frac{\sqrt{2}}{|\lambda_j-\lambda_k|} \leq \frac{\sqrt{2}}{b_1}
\end{align}
due to the fact that $|\exp(i\theta)-1|^2=(1-\cos\theta)^2+\sin^2\theta$, for all $\theta\in\mathbb{R}$ and Theorem \ref{asymp}. 
We now estimate the remainder terms of \eqref{sigma}:
\begin{align}
R=\frac{1}{T}\int\limits_0^T\int\limits_{\omega}|\sum\limits_{j=N+1}^{\infty}y_j(t,x)|^2\,dx\,dt
\end{align}
and
\begin{align}
\delta=\frac{1}{T}\left( \int\limits_0^T\int\limits_{\omega}\sum\limits_{j=1}^{N}y_j(t,x)\sum\limits_{k=N+1}^{\infty}\overline{y_k(t,x)}\,dx\,dt\right).
\end{align}
Using the fact that the $\phi_j's$ form a Hilbert basis, 
\begin{align}
&R\leq \frac{1}{T}\int\limits_0^T\int\limits_{\Omega}|\sum\limits_{j=N+1}^{\infty}y_j(t,x)|^2\,dx\,dt=\\&
\frac{1}{T}\sum\limits_{j=N+1}^{\infty}\int\limits_0^T\lambda_j^2|c_j\exp(i\lambda_jt)|^2\,dt=
\frac{1}{T}\sum\limits_{j=N+1}^{\infty}T\lambda_j^2|c_j|^2=\sum\limits_{j=N+1}^{\infty}\lambda_j^2|c_j|^2.
\end{align}
To bound $\delta$,
\begin{align}
T|\delta|=|\sum\limits_{j=1}^{N}\sum\limits_{k=N+1}^{\infty}\lambda_jc_j\lambda_k\overline{c_k}\alpha_{jk}| \leq \sum\limits_{j=1}^{\infty}(\lambda_jc_j)^2\left(\sum\limits_{k=N+1}^{\infty}\lambda_k^2c_k^2\alpha_{jk}^2\right)\leq \sum\limits_{k=N+1}^{\infty}\left(\max_{j}\alpha_{jk}\lambda_kc_k\right)^2,
\end{align}
whenever the normalization $\sum\limits_j(\lambda_jc_j)^2=1$ is used. By Parseval's theorem, since $u,\partial_t^2u\in L^2(\Omega)$, for every $\varepsilon>0$, there exists an $N\geq N(\varepsilon)$ such that
\begin{align}
\sum\limits_{j=N+1}^{\infty}\lambda_j^2|c_j|^2
\leq \varepsilon.
\end{align}
We conclude that for sufficiently large $N$,
\begin{align}
|R+2\delta|\leq \varepsilon\left(1+\frac{4}{Tb_1}\right).
\end{align}
Since $\varepsilon$ was arbitrary and $T\rightarrow \infty$, the theorem is proved. The corollary follows since, due to the assumption that 
$||\partial_tu(0,x)||_{L^2(\Omega)}=1$, we have
\begin{equation}
\inf_{\sum\limits_{j=1}^N\lambda_j^2|c_j|^2=1} \sum\limits_{j=1}^N\lambda_j^2|c_j|^2\int\limits_{\omega}\phi_j^2(x)\,dx= \inf_{1,...N} \int_{\omega} \phi_j^2(x)dx.
\end{equation}

 Note that in the case of non-simple eigenvalues, one can group the diagonal terms to obtain the desired result. This proves Theorem \ref{main1} and Corollary \ref{cmain1}. 
\end{proof}

\section{Basic Perturbation Theory}
In this section, we give an explicit example of how to calculate the eigenvalues and eigenvectors of the perturbed operator $H=-\Delta+\varepsilon V_0$, with \textbf{simple} eigenvalues $\lambda_n$. (Recall this means the eigenvalues have multiplicity 1). In the next section, more advanced results from \cite{kato} will be used to analyze the error terms. 

Let $H_0=-\Delta$ denote the standard Laplacian with eigenvalues $\lambda_{n0}$. There exists a corresponding basis $\phi_{n0}(x)$ such that 
\begin{align}
-\Delta\phi_{n0}(x)=\lambda_{n0}\phi_{n0}(x).
\end{align}
The following lemma relates the eigenvalues and eigenvectors of $H$ to those of $H_0$:
\begin{lem}
The eigenvalues to $H$ are given by:
\begin{align}\label{lambdaexpand}
\lambda_n=\lambda_{n0}+\varepsilon\left(\frac{\int V_0(x)\phi^2_{n0}(x)\,dx}{\int\phi_{n0}^2(x)\,dx}\right)+\mathcal{O}(\varepsilon^2).
\end{align}
The eigenfunctions to $H$ are given by:
\begin{align}
\phi_{n}(x)=\phi_{n0}(x)+\varepsilon\left(\sum\limits_{n\neq m}\left(\frac{\int \phi_{n0}\overline{V_0\phi_{m0}}\,dx}{\lambda_{n0}-\lambda_{m0}}\right)\phi_{m0}(x)\right)+\mathcal{O}(\varepsilon^2).
\end{align}
Here the $\mathcal{O}$ terms are uniform in $n$ depending on $\Omega$ and the $L^{\infty}(\Omega)$ norm of $V_0$. In particular we have that 
\begin{align}
||\phi_{n}(x)-\phi_{n0}(x)||_{L^2(\Omega)}\leq \varepsilon C_2(V_0,\Omega)||\phi_{n0}(x)||_{L^2(\Omega)}
\end{align}
and
\begin{align}
||\phi_{n}(x)-\phi_{n0}(x)-\varepsilon\left(\sum\limits_{n\neq m}\left(\frac{\int \phi_{n0}\overline{V_0\phi_{m0}}\,dx}{\lambda_{n0}-\lambda_{m0}}\right)\phi_{m0}(x)\right)||_{L^2(\Omega)}\leq \varepsilon^2C_3(V_0,\Omega)||\phi_{n0}(x)||_{L^2(\Omega)} 
\end{align}
where $C_2(V_0,\Omega)$, $C_3(V_0,\Omega)$ depend only on the geometry of $\Omega$, and the $L^{\infty}(\Omega)$ norm of $V_0$.
\end{lem}
 We do not prove the Lemma here, it is a result of \cite{kato} (see equation (II-3.39) in Example 3.6, where the constant is given explicitly), we only give an idea of why it is true. One will see the results in the next section are more general. If we make the approximation
\begin{align*}
&\phi_n(x)=\phi_{n0}(x)+\varepsilon\phi_{n1}(x)+\varepsilon^2\phi_{n2}(x)+. . .\\&
\lambda_n=\lambda_{n0}+\varepsilon\lambda_{n1}+\varepsilon^2\lambda_{n2}+. . . ,
\end{align*}
then it follows by substitution that
\begin{align*}
&(-\Delta+\varepsilon V_0(x))(\phi_{n0}(x)+\varepsilon\phi_{n1}(x)+\mathcal{O}(\varepsilon^2))=\\&(\lambda_{n0}+\varepsilon\lambda_{n1}+\mathcal{O}(\varepsilon^2))(\phi_{n0}+\varepsilon\phi_{n1}(x)+\mathcal{O}(\varepsilon^2)).
\end{align*}
Equating the leading order terms,
\begin{align*}
(-\Delta-\lambda_{n0})\phi_{n0}(x)=0.
\end{align*}
At order $\varepsilon$, we have
\begin{align}\label{eps}
(-\Delta-\lambda_{n0})\phi_{n1}+(V_0(x)-\lambda_{n1})\phi_{n0}(x)=0.
\end{align}
The desired result for computing the first terms follows by taking the inner product of \eqref{eps} with $\phi_{j0}$ for $j\neq n$. We have to have a way of encoding this inductive process of matching up the terms. In the next section we introduce the operators $S$ and $P$ which allow us to do just that. The terms are computed for $L^2(\Omega)$ eigenfunctions, but the analysis is more sophisticated because when computing the result of the matching over $L^2(\omega)$, one loses the orthogonality of the eigenfunctions over the region of integration. 

We have the following example of an operator with simple eigenvalues. 
\begin{ex}
We consider the eigenvalue problem with $\alpha>1$
\begin{align}
-u_{\varepsilon}''+\varepsilon x^{-2\alpha}u_{\varepsilon}=\lambda_{\varepsilon}u_{\varepsilon} \qquad u_{\varepsilon}(0)=u_{\varepsilon}(1)=0
\end{align}
the unperturbed problem is 
\begin{align}
-u''=\lambda u\qquad u(0)=u(1)=0
\end{align}
with simple eigenvalues $\lambda=n^2\pi^2$, with $n=1,2,3. . .$ with corresponding normalised eigenfunctions $u=2^{1/2}\sin(n\pi x), n=1,2,...$ The quadratic form associated with the potential $Q_1(f)$ with domain $\{f\in L^2(0,1): x^{-\alpha}f\in L^2(0,1)\}\subset H^1_0(\Omega)$ is closed in $L^2(0,1)$. The unperturbed operator is \emph{stable} with respect to perturbations \cite{kato}. This is the assumption on both of the main theorems (Theorem 2.3 and Theorem 2.4). This example is from \cite{Greenlee}. Stability of $\lambda$ for the unperturbed problem means that for $\varepsilon$ sufficiently small, the intersection of any isolating interval for $\lambda$ and the spectrum of the perturbed operator consists only of simple eigenvalues. The unperturbed/perturbed operator pair here satisfies the criterion of Theorem 5.1.12 in \cite{kato} for stability which holds provided the left hand side \eqref{expression} is smaller than $1/2$, which is true for sufficiently small $\varepsilon$. This also applies to the first numerical example in the appendix. Usually stability is automatically satisfied when $-\Delta$ has simple spectrum and $\varepsilon_0$ is sufficiently small, c.f. Lemma 2.1 in \cite{G}. 
\end{ex}
As a general remark on the example and computations above, the difficulty lies in quantifying the error terms which are usually formulated in the sense of $L^2(\Omega)$ not $L^2(\omega)$, which is why the next section is required. 

We have the following result for more general Riemannian metrics which shows that the assumption of simple spectrum in our case covers generic domains $\Omega$. Symmetry usually destroys the assumption of spectral simplicity c.f. \cite{Uh} and this is also discussed in Section 7. 
\begin{thm}[\cite{BW} and \cite{Uh}]
Let $\mathcal{M}$ be a compact manifold of dimension greater than 1 and $\mathcal{C}$ a conformal class of Riemannian metrics of fixed volume on $\mathcal{M}$. Given $k\geq 1$ and $d\geq 2$, the subset of $\mathcal{C}$ of metrics for which the $k^{th}$ eigenspace is of dimension $d$ is a sub-manifold of codimension at least 1. In particular, the subset of $\mathcal{C}$ of metrics admitting a non-simple eigenvalue of the Laplacian is a countable union of submanifolds of codimension at least $1$. 
\end{thm}
This Theorem asserts that for a given compact manifold $\mathcal{M}$ "most" Riemannian metrics $g$ on $\mathcal{M}$ are simple, meaning the eigenspace of the Laplace operator $\Delta_g$ is one-dimensional and that this set is pathwise connected. Her proof naturally remains true for $0^{th}$ order perturbations, like the ones we have here c.f. \cite{G}. We leave the question of what happens to the observability constants for metric perturbations to future work. 

\section{Advanced Perturbation Theory}\label{advanced}

In this section, we elaborate on advanced perturbation theory for a better understanding of the results derived in the paper. Let $X$ be an arbitrary Hilbert space, as in \cite{kato}, and $R(A)$ be the range of the bounded operator $A$. The monograph \cite{kato} by Kato computes perturbation theory results for generic bounded operators $A$, and since our operator satisfies the conditions in \cite{kato} for a Type (A) holomorphic operator in the parameter $\varepsilon$ (Theorem 2.6 of \cite{kato}), the perturbation theory derived in the book applies. 

Let $P$ be the projection operator, $\lambda$ be one of the eigenvalues of $H_0=-\Delta$, $H=-\Delta+ \varepsilon V_0= -\Delta+ V$, and let $\lambda_k$, $P_k$, $k=1,2,...$, be the eigenvalues and eigenprojections of $H_0=-\Delta$ different from $\lambda$ and $P$ under consideration. Let $\{x_1, . . ,x_m\}$ denote a basis of $M=R(P)$, and $\{x_{k1}, . . ,x_{km_k}\}$ denote a basis of $M_k=R(P_k)$ for each $k$. The union of the vectors $x_j$ and $x_{kj}$ forms a basis of $X$ consisting of eigenvectors of $H_0=-\Delta$ and is adapted to $X=M\oplus M_1\oplus . . .$ of $X$. The adjoint basis of $X^{\star}$ is adapted to $X^{\star}=M^{\star}\oplus M_1^{\star}\oplus . . .$, where $M^{\star}=R(P^{\star})$, $M_1^{\star}= R(P_1^{\star})$, etc.
Let $\{e_1, . . . ,e_m\}$ denote the adjoint basis of $M^*$, and $\{e_{k1}, . . . ,e_{km_h}\}$ denote the basis of $M_k^*$ for $k=1,2,....$

For any $u\in X$, 
\begin{align}
Pu=\sum\limits_{j=1}^m\langle u,e_j\rangle x_j\quad P_ku=\sum\limits_{j=1}^{m_k}\langle u,e_{kj}\rangle x_{kj},\quad \forall k=1,2. . .
\end{align}
We define the operator $S$ as the value of the reduced resolvent of $H_0=-\Delta$, such that  $SP=PS=0$, and $(H_0-\lambda)S=S(H_0-\lambda)=1-P$.
The $P_k's$ are the orthogonal projections such that 
\begin{align}
P=\sum\limits_{k=1}^{\infty}P_k,
\end{align} 
and moreover, by definition, $P^2=P$. For $\lambda$ in our particular eigenspace, it follows that one can write the operator $S$ explicitly as 
\begin{align}
&Su=\sum\limits_k(\lambda_k-\lambda)^{-1}P_ku 
=\sum\limits_{k,j}(\lambda_k-\lambda)^{-1}\langle u, e_{k j}\rangle x_{kj}.
\end{align}
using the definitions (I-5.32) and Section II.2 in \cite{kato} . 
If we expand $\lambda_{\varepsilon}$, which is an eigenvalue of $H$, in a perturbation series as
\begin{align}
\lambda_{\varepsilon}=\lambda+\varepsilon\hat{\lambda}^1+\varepsilon^2\hat{\lambda}^2+. . . ,
\end{align}
one obtains the following expressions for the expansions of the eigenvalues $\hat{\lambda}^n$ (II-(2.35) \cite{kato}):
\begin{align}
&\hat{\lambda}^1=\frac{1}{m}\sum\limits_j\langle V_0x_j,e_j\rangle. \\&
\hat{\lambda}^2= -\frac{1}{m}\sum\limits_{i,j,k}(\lambda_k-\lambda)^{-1}\langle V_0x_1,e_{kj}\rangle \langle V_0x_{kj},e_i\rangle.
\end{align}
Suppose that the eigenvalue of $\lambda$ of $H_0$ is simple, implying that $m=1$. To derive an expansion for a particular eigenvector eigenvalue pair, one can set $x_1=\phi_{j0}(x)$ and $e_1=\overline{\phi}_{j0}(x)$ as in the last section. (Now $j$ just refers to the index of the eigenfunction, a distinct index from the one above) The operators $P_j$ and $S_j$ can be written as 
\begin{align}\label{decomp}
P_ju=\langle u,\phi_{j0}\rangle_{L^2(\Omega)} \phi_{j0}, \qquad Su=\sum\limits_{j\neq k}\frac{P_ku}{\lambda_{j0}-\lambda_{k0}}.
\end{align}
This substitution compares immediately to the results in the previous section for the expansion of the eigenvalues \eqref{lambdaexpand}.
Now we describe a more advanced decomposition of the eigenvectors.  

Assuming for simplicity that $m=1$, a convenient form of the eigenvector $\phi_j$ of $H=-\Delta+ \varepsilon V_0$ corresponding to the eigenvalue $\lambda_{\varepsilon}$ is given by
\begin{align}\label{projection}
\phi_j=\langle P_j(\varepsilon)\phi_{j0},\overline{\phi_{j0}}\rangle_{L^2(\Omega)}^{-1}P_j(\varepsilon)\phi_{j0},
\end{align}
where $\phi_{j0}$ is the unperturbed operator of $H_0$ for the eigenvalue $\lambda$ and $\overline\phi_{j0}$ is the eigenvector of the adjoint operator $H_0^{\star}$. $P_j(\varepsilon)$ is the projection onto the $j^{th}$ eigenspace of $H$. The assumption of stability here is used in a hidden way as we want to make sure the projection onto the eigenspace is well-defined. In particular the projection is defined the integral of the resolvent over an interval containing only one eigenvalue. As such, in order for the projection to be well-defined, the eigenvalue needs to be sufficiently isolated, whence the assumption of simplicity in a perturbed neighbourhood of $-\Delta$.  We refer the reader to Theorem 5.1.12 in \cite{kato}, and Lemma 2.1 in \cite{G} for a precise description of $\varepsilon_0$, the threshold required. In the case of non-simple eigenvalues the representation above would depend on more that one $\phi_{j0}$, which would be difficult to analyze. We suppress the subscript $j$ in the operators $P$ and $S$ where it is understood. This gives rise to the following normalization conditions: 
\begin{align}
\langle \phi_j,\overline{\phi_{j0}} \rangle_{L^2(\Omega)}= 1,\quad  \langle \phi_j-\phi_{j0},\overline{\phi_{j0}}\rangle_{L^2(\Omega)}=0, \quad P(\phi_j-\phi_{j0})=0.
\end{align}
The relation $(H-\lambda_{\varepsilon})\phi_{j}=0$ can be re-written as
\begin{align}\label{above}
(H_0-\lambda)(\phi_{j}-\phi_{j0})+(V-\lambda_{\varepsilon}+\lambda)\phi_{j}=0,
\end{align}
where $A=H-H_0=\varepsilon V_0(x)=V$. Multiplying \eqref{above} from the left hand side by $S$ and noting that $S(H_0-\lambda)=1-P$, 
\begin{align}
\phi_j-\phi_{j0}+S[V-\lambda_{\varepsilon}+\lambda]\phi_{j}=0.
\end{align}
Moreover, as $S\phi_{j0}=0$ and writing $\phi_j=\phi_j-\phi_{j0}+\phi_{j0}$ in the last term above, one gets
\begin{align}\label{exp}
&\phi_{j}-\phi_{j0}=-(1+S(V-\lambda_{\varepsilon}+\lambda))^{-1}SV\phi_{j0}\\&= \nonumber-S(1+VS-(\lambda_{\varepsilon}-\lambda)S_{\alpha})^{-1}V\phi_{j0},
\end{align}
for sufficiently small $\varepsilon$, with $S_{\alpha}=S-\alpha P$ and $\alpha$ is an arbitrary scalar. Equation \eqref{exp} is formula (II-3.29) in \cite{kato}.

One can then compute
\begin{align}\label{innerS}
\langle \phi_j-\phi_{j0}, \phi_{j0}\rangle_{L^2(\omega)}=\langle -S(1+VS-(\lambda_{\varepsilon}-\lambda)S_{\alpha})^{-1}V\phi_{j0},\phi_{j0}\rangle_{L^2(\omega)}.
\end{align}

The asymptotics for the scalar $\lambda_j$ are well worked out for small $\varepsilon$. Let $q=||V_0S||$, $s_0=||S||$, $p=||V_0P||$, $s=||S-\alpha P||$ for any $\alpha$, where we use the operator norm.  A subscript will denote the set over which the operator norm is taken. 

For a linear operator $A$ acting on $H^1_0(\Omega)$, we let $||A||_0$ denote the norm
\begin{align}
||A||_0=\left \{\sup_j\langle Au,\phi_{j0} \rangle\,\, \textrm{such that} \,\, u: \sup_j\langle u,\phi_{j0} \rangle=1\right\}.
\end{align}

Set $p,s,q$ to have norm $||\cdot||_0$ and define
\begin{align}
\Psi(\varepsilon)=\left((1-(ps+q)\varepsilon)^2-4ps\varepsilon^2\right)^{1/2}.
\end{align}
As a result
\begin{align}
\label{expression}
|\lambda_{j}-\lambda_{j0}-\varepsilon\hat{\lambda}^1|=|\lambda_{\varepsilon}-\lambda-\varepsilon\hat{\lambda}^1|\leq \frac{2pq\varepsilon^2}{1-(ps+q)\varepsilon+\Psi(\varepsilon)},
\end{align}
which is formula (II-3.18) in \cite{kato}, with the norm $||\cdot||_0$. The expansion \eqref{innerS} derived above is given in section II and exercise II-3.16 in the monograph by Kato \cite{kato}. 

Now the difficulty comes in computing inner products of $\phi_j-\phi_{j0}$ over the smaller sets $\omega$ where one loses the powerful orthogonality conditions. We recall following well-known Lemma on von Neumann series
\begin{lem}\label{VN}
Let $A: X\rightarrow X$ be a linear operator on the Banach space $X$. We then have
\begin{align}
\sum\limits_{j=0}^{\infty}A^ju=(Id-A)^{-1}u \quad \forall u\in X
\end{align}
provided 
\begin{align}
||A^ju||_{X}\leq \delta_1^j||u||_{X}\quad\quad \forall j\in \mathbb{N}
\end{align}
with $\delta_1\in (0,1/2)$.
\end{lem}
c.f. Lemma 2.1 in \cite{conway} 

In order to compute \eqref{innerS}, we want to use the Lemma \ref{VN} to essentially find a convergent von-Neumann series for 
\begin{align}
\label{operator}
(1+VS-(\lambda_{\varepsilon}-\lambda)S_{\alpha})^{-1}
\end{align}
with \eqref{expression} so that we may obtain precise bounds on the rate of decay of the inner products $\langle\phi_{j}-\phi_{j0},\phi_{j0}\rangle_{L^2(\omega)}$. These arguments are rather delicate as we are not integrating over the whole $\Omega$. We let $M_0$ denote a generic constant that depends on the volume of $\Omega$ and $||V_0||_{L^{\infty}(\Omega)}$ We state the four necessary Lemmas first, followed by their technical proofs to see how the pieces fit together to allow us to use Lemma \ref{VN} by examining each term in the series expansion to bound \eqref{operator}. 

\begin{lem}\label{bound2}
With no assumptions on the support of the potential, we have the following estimate for $u_{\omega}\in L^2(\omega)$ with $\mathrm{supp}(u_{\omega})\subseteq \omega$
\begin{align}
|\langle V(Su_{\omega}),u_{\omega}\rangle_{L^2(\omega)}|=|\sum\limits_{j\neq k}\frac{\langle u_{\omega},\phi_{k0}\rangle_{L^2(\Omega)}\langle V\phi_{k0},u_{\omega}\rangle_{L^2(\omega)}}{\lambda_{j0}-\lambda_{k0}}|\leq \varepsilon M_0||u_{\omega}||_{L^2(\omega)}^2.
\end{align}
\end{lem}
Let $A_v$ be the linear operator defined as multiplication by 
\begin{align}
((\lambda_{j0}-\lambda_{j})-V).
\end{align}

\begin{lem}\label{bound1}
With no assumptions on the support of the potential, we have the following estimate for $u_{\omega}\in L^2(\omega)$ with $\mathrm{supp}(u_{\omega})\subseteq \omega$ for all $N\geq 0, N\in\mathbb{N}$,
\begin{align}
|\langle(A_vS)^N(u_{\omega}),u_{\omega}\rangle_{L^2(\omega)}|\leq (\varepsilon M_0)^N||u_{\omega}||_{L^2(\omega)}^2.
\end{align}
\end{lem}

\begin{lem}\label{inversion}
If $supp(V_0)\subset \omega$ and $\alpha\in (0,1)$, then the operator $S(1-A_vS-\alpha(\lambda_j-\lambda_{j0})P)^{-1}V$ is bounded $L^2(\Omega)\mapsto L^2(\omega)$.
\end{lem}

\begin{lem}\label{bound3}
There is a choice of $\varepsilon$ sufficiently small, such that for all $\delta\in (0,1/2)$, the following inequality holds:
\begin{align}
|\langle \phi_{j0}-\phi_j,\phi_{j0}\rangle_{L^2(\omega)}|=|\langle S(1+VS-(\lambda_{\varepsilon}-\lambda)S_{\alpha})^{-1}(V\phi_{j0}),\phi_{j0}\rangle_{L^2(\omega)}|\leq \delta ||\phi_{j0}||^2_{L^2(\omega)}.
\end{align}
\end{lem}

\begin{proof}[Proof of Lemma \ref{bound2}]

By the Cauchy Schwartz inequality we have,
\begin{align}\label{cs}
& |\langle VSu_{\omega},u_{\omega}\rangle_{L^2(\omega)}|\leq \\& \sum\limits_{j\neq k}\frac{|\langle u_{\omega},\phi_j\rangle_{L^2(\Omega)}\langle u_{\omega},V\phi_j\rangle_{L^2(\omega)}|}{|\lambda_j-\lambda_k|}\leq \left(\sum\limits_{j\neq k}\frac{|\langle u_{\omega},\phi_j\rangle_{L^2(\Omega)}|^2}{|\lambda_j-\lambda_k|^2}\right)^{1/2}\left(\sum\limits_{j\neq k}|\langle u_{\omega},V\phi_j\rangle_{L^2(\omega)}|^2\right)^{1/2} \nonumber
\end{align}
We know from Theorem \ref{asymp} from \cite{Davies}, that 
\begin{align}
|\lambda_{j0}-\lambda_{k0}|>C, \quad j\neq k 
\end{align}
where $C$ depends on $\Omega$ independent of the index set. The constant exists because all of the eigenvalues are simple and isolated. However in practice for dimensions higher than 2 the size of $C$ maybe difficult to ascertain. The desired result follows immediately from Parseval's theorem, noting that $u_{\omega}$ and $Vu_{\omega}$ are $L^2(\Omega)$ functions. We remark that this is where we use the assumption $\mathrm{supp} V\subset \omega$ later for the main proof as the inner product $<u_{\omega},\phi_j>$ is over $L^2(\Omega)$ which cannot be bounded by $||u_{\omega}||_{L^2(\Omega)}$ unless $u_{\omega}$ has compact support in $\omega$. 
\end{proof}

\begin{proof}[Proof of Lemma \ref{bound1}]
The bound for this inner product is constructed inductively as
\begin{align}
(A_vS)^N(u_{\omega})=
\sum\limits_{m_0\neq j}\frac{\langle u_{\omega},\phi_{m_0}\rangle_{L^2(\Omega)}}{\lambda_{j0}-\lambda_{m_0}}\sum\limits_{m_1\neq j}\frac{\langle A_v\phi_{m0},\phi_{m_1}\rangle_{L^2(\Omega)}}{\lambda_{j0}-\lambda_{m_1}} . . . \sum\limits_{m_{N-1}\neq j}\frac{\langle A_v\phi_{m_{N-2}},\phi_{m_{N-1}}\rangle_{L^2(\Omega)}}{\lambda_{j0}-\lambda_{m_{N-1}}}A_v\phi_{m_{N-1}} 
\end{align}
We then use the proof of the previous Lemma \ref{bound2}, but with \eqref{cs}  applied to each of the cross terms:
\begin{align}
\sum\limits_{m_{i-1}\neq j}\frac{\langle A_v\phi_{m_{i-2}},\phi_{m_{i-1}}\rangle_{L^2(\Omega)}}{\lambda_{j0}-\lambda_{m_{i-1}}} \quad i=2,  . . ,N
\end{align}
 to reach the desired conclusion, noting that $||A_v||_{L^{\infty}(\Omega)}$ is almost equivalent to $||V||_{L^{\infty}(\Omega)}$. Alternatively we know for bounded operators $A, B$ with $A: X\rightarrow X, B: X\rightarrow X$, $X$ a Hilbert space, that $||A B||_{op}\leq ||A||_{op}||B||_{op}$ which when applied to $A=V$ and $B=S$ from the previous lemma, gives the result as well.  
\end{proof}

\begin{proof}[Proof of Lemma \ref{inversion}]
The mapping $A_{v}S$ satisfies all the properties of Lemma \ref{VN} with $X$ the space $L^2(\Omega)$ restricted to the functions with compact support in $\omega$, by Lemma \ref{bound1}. The space $X$ is $L^2_0(\omega)$, which is a Hilbert space (although most people are more familiar with $H_0^1(\omega)$). Therefore since $Vu\in X$, for all $u\in L^2(\Omega)$, we are done. 
\end{proof}

\begin{proof}[Proof of Lemma \ref{bound3}]
The inner product using \eqref{innerS} and Lemmas \ref{bound1} and \ref{VN} with $A_vS+\alpha(\lambda_j-\lambda_{j0})P$, and $\delta_1=\varepsilon M_0$ is bounded as
\begin{align}
||\phi_{j0}||_{L^2(\omega)}^2((\varepsilon M_0)+(\varepsilon M_0)^2+ . . .)=\frac{\varepsilon M_0}{1-\varepsilon M_0}||\phi_{j0}||_{L^2(\omega)}^2.
\end{align}
If $\varepsilon M_0$ is chosen to be sufficiently small, one obtains:
\begin{align}
\frac{\varepsilon M_0}{1-\varepsilon M_0}<\delta
\end{align}
with $\delta\in (0,1/2)$, implying $\varepsilon<\frac{\delta}{2M_0}$. We use the fact $\mathrm{supp}V_0\subset\omega$, which makes the function $V\phi_j$ $\forall j$ have support in $\omega$. 
\end{proof}

\begin{rem} Some of the could be extended to the case of non-simple eigenvalues and other Hermitian operators using perturbation theory found in \cite{kato}, but we focus on simple eigenvalues for ease and clarity. 
\end{rem}

\section{Proof of Theorems \ref{perturbationtheorem} and Theorem \ref{ngpc} for Convergence Estimates} \label{peturbmsec}
\begin{proof}[Proof of Theorem \ref{perturbationtheorem}]
We recall that eigenfunctions of $-\Delta$ and $-\Delta+\varepsilon V_0$ ($\varepsilon$ sufficiently small) with Dirichlet boundary conditions are real-analytic in $\Omega$. We can then view 
$\int\limits_{\omega}\phi_j^2(x)\,dx=f_V(j)$ as a function of $\mathbb{N}$ taking values in $(0,1]$ and similarly for $\int\limits_{\omega}\phi_{j0}^2(x)\,dx=f(j)$. We only need to show that the following inequality is true for some order terms independent of the index $j$:
\begin{align}\label{ineqvar}
\int\limits_{\omega}\phi_j^2(x)\,dx=(1+\mathcal{O}(\epsilon))\int\limits_{\omega}\phi_{j0}^2(x)\,dx.
\end{align}
These terms will bound the deviation from the original constant when including the potential term, and taking the $\inf\limits_j$ of the inequality gives the desired result. We need control over the order $\epsilon$ terms and show they are \emph{uniformly} bounded, independent of the $\phi_{j0}$, e.g., the order terms are smaller than $1/2$ for sufficiently small $\epsilon$. If we can show this inequality, we will arrive at
\begin{align}
\frac{1}{2} f_V(j)\leq f(j)\leq \frac{3}{2}f_V(j).
\end{align}
Taking the infimum over $j$ gives the desired conclusion. 

 By the perturbation theory estimates in Section \ref{advanced}, in Lemma \ref{bound1} and Lemma \ref{bound3}, by using the Lemma \ref{VN}, the terms  
\begin{align}
2|\langle\phi_{j0}-\phi_j,\phi_{j0}\rangle_{L^2(\omega)}|+||\phi_{j0}-\phi_j||^2_{L^2(\omega)}
\end{align}
which contribute to the order terms in \eqref{ineqvar} are bounded as in Lemma \ref{bound3}. Therefore, it suffices to pick $\varepsilon$ as in Lemma \ref{bound3}: $3\delta<1/2$, in order to obtain \eqref{ineqvar}, with $\mathcal{O}(\epsilon)$ terms less than $1/2$.
\end{proof}

We now take a moment to remark on why the assumption of simplicity in Theorem \ref{perturbationtheorem} is spectrally sharp, by outlining a counter-example from the details of \cite{MR2}.  In Remark 2.2 of \cite{MR2}, on the sphere $\mathbb{S}^2$ they construct a potential of arbitrarily small support and size, and an open set $\omega\subset\mathbb{S}^2$ such that
\begin{align}
\inf\{ \int\limits_{\omega}\phi^2(x)\,dx \quad \phi\quad \textrm{eigenfunction of} -\Delta \quad \textrm{s.t.}\,\, ||\phi||_{L^2(\mathbb{S}^2)}=1\} =0
\end{align}
however
\begin{align}
\inf\{ \int\limits_{\omega}\psi^2(x)\,dx \quad \psi\quad \textrm{eigenfunction of} -\Delta+V \quad \textrm{s.t.}\,\, ||\psi||_{L^2(\mathbb{S}^2)}=1\} >0
\end{align}
One can take this a step further and construct a sequence of normalised eigenfunctions of $-\Delta$ such that $\phi_j$ belongs to the eigenspace associated to the eigenvalue $j(j+1)$ in such a way that 
\begin{align}
\lim\limits_{j\rightarrow\infty}\int\limits_{\omega}\phi_j^2(x)\,dx=0
\end{align}
as in \cite{Macia}. One can complete this sequence to obtain an orthonormal basis of $L^2(\mathbb{S}^2)$ consisting of eigenfunctions of $-\Delta$ such that $J^0(\chi_{\omega})=0$. On the other hand, $J^V(\chi_{\omega})>0$. In this particular case where the spectrum is non simple, one cannot obtain the strong iff statement in Theorem \ref{perturbationtheorem} because the representation for the perturbed eigenfunctions in \eqref{projection} depends on \textbf{all} of the $2j+1$ eigenfunctions of $-\Delta$ associated to each eigenvalue $j(j+1)$, introducing cross terms in Lemmas 6.2 and 6.3 which are computationally difficult to control. If the perturbation theory was carried out for this non-simple case which is possible, then at most one could conclude the constants are close. 

\begin{proof}[Proof of Theorem \ref{ngpc}]
Let $\tilde{\phi}=\phi_{j0}-\phi_j$, then we can write for any $j$ and $a(x)$ 
\begin{align}
\int\limits_{\Omega}a(x)\phi_j^2\,dx=\int\limits_{\Omega}a(x)\phi_{j0}^2\,dx+\int\limits_{\Omega}a(x)(2\phi_{j0}\tilde{\phi}+\tilde{\phi}^2)\,dx.
\end{align}
As we have that using the normalisation condition $||\phi_{j0}||^2_{L^2(\Omega)}=1$, and the "observation region" is the entirety of $\Omega$ 
\begin{align}\label{meq}
|\int\limits_{\Omega}a(x)(2\phi_{j0}\tilde{\phi}+\tilde{\phi}^2)\,dx|\leq 3|\int\limits_{\Omega}a(x)\tilde{\phi}^2\,dx|\leq 3L|\Omega|C_2^2(V_0,\Omega)\varepsilon^2.
\end{align}
We obtain \begin{align}\label{mmin}
|\int\limits_{\Omega}a(x)\phi_j^2\,dx-\int\limits_{\Omega}a(x)\phi_{j0}^2\,dx|\leq 3L|\Omega|C_2^2(V_0,\Omega)\varepsilon^2
\end{align}
where the order terms are uniformly bounded where we have used Lemma 5.1. We set $C_1(V_0,\Omega)=3|\Omega|C_2^2(V_0,\Omega)$. Notice the orthogonality relations imply no additional regularity is needed on $V_0$.  
Without loss of generality assume $J_{\varepsilon}(a)-J(a)>0$ then we have that
\begin{align}
J_{\varepsilon}(a)+\inf_j\left(\int\limits_{\Omega}a(x)\phi_{j0}^2\,dx-\int\limits_{\Omega}a(x)\phi_{j}^2\,dx\right)\leq J(a).
\end{align}
Re-arranging we obtain a bound on $J_{\varepsilon}(a)-J(a)$, depending on $|\Omega|$ and the $L^{\infty}(\Omega)$ norm of the potentials as desired, after using \eqref{meq}. Notice that this is probably the best control of the errors as $-\inf_j(A(j))=\sup_j(-A(j))$ for all functionals $A(j)$. 

With out loss of generality we assume 
$\max_{a\in \overline{M}_L}J_{\varepsilon}(a)-\max_{a\in \overline{M}_L}J(a)>0$, and we obtain
\begin{align}
\max_{a\in \overline{M}_L}J_{\varepsilon}(a)\leq \max_{a\in \overline{M}_L}\left(J_{\varepsilon}(a)-J(a)\right)+\max_{a\in \overline{M}_L}J(a)\leq C_1(V_0,\Omega)\varepsilon^2+\max_{a\in \overline{M}_L}J(a)
\end{align} 
with constant given to us by \eqref{meq} and Lemma 6.5. 
\end{proof}

\section{Numerics and Examples}

 This section presents the results of our numerical experiments. We examine the cases of the unit interval and the unit disk. The convergence issues for the functionals in question are discussed in the Appendix. 
   
  

\subsection{Interval [0,1]}

         The first experiment involved $\Omega=[0,1]$. The orthonormal eigenvectors of $-\Delta$ with Dirichlet boundary conditions on this domain are 
  $f_n(x)= \sqrt{2} \sin(n\pi x)$, for $n=1,2,...$, with eigenvalues of $\lambda_n= \pi^2 n^2$, for $n=1,2,...$. The eigenvalues have multiplicity one.
  
   Next, we calculate the eigenvectors and eigenvalues of $H= -\Delta+\varepsilon V_0$ on the unit interval with Dirichlet boundary conditions. According to perturbation theory of Section 5, the eigenvalues of operator $H$  are given by:
\begin{align}
\lambda_n=\lambda_{n0}+\varepsilon\left(\frac{\int V_0(x)\phi^2_{n0}(x)\,dx}{\int\phi_{n0}^2(x)\,dx}\right)+\mathcal{O}(\varepsilon^2).
\label{eigenvalue}
\end{align}
The eigenfunctions of $H$ are given by:
\begin{align}
\phi_{n}(x)=\phi_{n0}(x)+\varepsilon\left(\sum\limits_{n\neq m}\left(\frac{\int \phi_{n0}\overline{V_0\phi_{m0}}\,dx}{\lambda_{n0}-\lambda_{m0}}\right)\phi_{m0}(x)\right)+\mathcal{O}(\varepsilon^2).
\label{eigenvector}
\end{align}
     In our case, we use the potential:
     \begin{equation}
         V_0= x^2 \chi_{[0.5-\delta,0.5+\delta]},
     \end{equation}
     where $\delta$ is a parameter $\in [0,0.5]$.

     Matlab was used to code the experiments. The integration of functions with explicit formulas was performed using the {\it integral} function in Matlab. We used a mesh size of about 1000 equal increments. The first two hundred eigenfunctions were calculated. 
     
     Next, we consider the problem of maximizing the functional $J_N^V(\chi_\omega)$, 
    \begin{align}
J_N^V(\chi_{\omega})=\inf_{1\leq j\leq N}\int_{\omega} \limits\phi_j^2(x)\,dx.
\end{align}     
     over all subsets satisfying $|\omega|= L|\Omega|$, for some $L \in (0,1)$. A subset with this property is called the optimal set. According to Proposition 4.1 of \cite{ptz}, in the case of the $-\Delta$ operator, the supremum of $J(\chi_\omega)=\inf\limits_{1\leq j\leq N}\int_{\omega} \limits\phi_{j0}^2(x)$ over $\mathcal{M}_L$, equal to L. When $L=0.5$, the supremum is reached for all measurable subsets $\omega$ of $[0,1]$ satisfying $|\omega|= 0.5|\Omega|$, such that $\omega$ and its symmetric image are complementary in [0,1]. Note that, for the $-\Delta$ operator, $\inf_{1\leq j\leq N}\int_{[0,0.5]} \limits\phi_{j0}^2(x)\,dx=0.5$ since
    \begin{align}
\int_0^{0.5} 2 \sin^2(n\pi x)\,dx= 0.5. \hspace{0.2cm} \forall n= 1,2,3,...
\end{align}        
     
     The more interesting case is the $H= -\Delta+\varepsilon V_0$ operator, and we address the question by using $L=0.5$ and computing $J_N^V(\chi_{\omega})$ for subsets satisfying $|\omega|= 0.5|\Omega|$ and the conditions of Proposition 4.1 of \cite{ptz}. In particular, we present results for $\omega= [0,5]$ and $N=200$. 

To calculate $J_N^V(\chi_{[0,0.5]})$ for $H= -\Delta+\varepsilon V_0$, integration using the left point and 1000 equal increments in [0,1] were used.  The $\delta$ and $\varepsilon$ variables were varied as shown in Table \ref{tbl:[0,1]}. The values in the table show that in all cases, the value of $J_N^V(\chi_{[0,0.5]})$ is very close to 0.5, which is the answer for the $-\Delta$ operator.

\begin{table}[!htb]
\begin{center}

\begin{tabular}{ | c | c | c | c | c | c | }
\hline
  $\varepsilon/\delta$  &       0.1.        &       0.2.           &         0.3.             &        0.4          &          0.475. \\
\hline
0.01       &     0.499997124   &  0.499984760    &  0.499972504   & 0.499968543  & 0.499968340  \\
\hline
0.05      &      0.499985619  &    0.499923804  &  0.499862542  & 0.499842757 &  0.499841748 \\
\hline
0.1        &      0.499971238  &   0.499847620   &   0.499725137  &  0.499685621  &  0.499683617 \\
\hline
0.5        &     0.499856202     &   0.499238531 &  0.498627808 &  0.498432406   &  0.498422979 \\
\hline
1          &      0.499712437 &    0.498478145  &   0.497260919  &  0.496875569  &  0.496858189  \\
\hline
\end{tabular}

\caption{Value of $J_N^V(\chi_{[0,0.5]})$}
\label{tbl:[0,1]}
\end{center}
\end{table}

\subsection{Unit Disk}

The orthonormal eigenvectors of $-\Delta$ on a unit disk with Dirichlet boundary conditions are given by the triply indexed sequence
\begin{equation}
     \phi_{j k m 0} = 
     \begin{cases}
      R_{jk}(r)/\sqrt{2\pi}, & \text{$if \hspace{0.2cm} j=0,$}  \\
     R_{jk}(r) Y_{jm}(\theta), & \text{$if \hspace{0.2cm} j \geq 1,$}   \\
        \end{cases}
\end{equation}
  for $j=0,1,2,...$, $k=1,2,...$ and $m=1,2$, where $(r,\theta)$ are polar coordinates. Here, $Y_{j1}(\theta)= \frac{1}{\pi}cos(j\theta)$, $Y_{j2}(\theta)= \frac{1}{\pi}sin(j\theta)$ and
   \begin{equation}
       R_{jk}(r)= \sqrt{2}\frac{J_j(z_{jk}r)}{|J'_j(z_{jk})|},
   \end{equation}
  where $J_j$ is the Bessel function of the first kind of order $j$, and $z_{jk}>0$ is the $k^{th}$ zero of $J_j$. The eigenvalues are given
  by the double sequence of $-z^2_{jk}$. Their multiplicity is 1 if $j=0$ and 2 if $j \geq 1$. 
  
  
     To compute the eigenvectors and eigenvalues of $H= - \Delta+\varepsilon V_0$, we use formulas \eqref{eigenvalue} and \eqref{eigenvector}. In this case the corresponding functionals are mock functionals as \eqref{eigenvalue} and \eqref{eigenvector} do not take into account the degeneracy of the problem which is that the eigenvalues are of multiplicity two. The correct formulae require some complicated normalisation constants given by \eqref{projection}.
     
 In our case, we use the potentials
     \begin{equation}
        V_0(r)= 1/r^2 \chi_{\{r \leq \delta\}} \hspace{0.2cm}  \text{and} \hspace{0.2cm}  V_0(r)= r \chi_{\{r \leq \delta\}},
     \end{equation}
     where $\delta <1$.
     
  There are several important equalities to note here. For radial subsets $\omega$ of the form $\omega= \{(r,\theta) \in [0,1] \times [0,2\pi] | \theta \in \omega_0\}$,
  \begin{equation}
        \int_{\omega} \phi_{jkm0}(x)^2 dx= \int_0^1 R_{jk}(r)^2r dr   \int_{\omega_0} Y_{jm}(\theta)^2 d\theta   =  \int_{\omega_0}   Y_{jm}(\theta)^2 d\theta,
  \end{equation}
  since $ \int_0^1 R_{jk}(r)^2r dr =1$.
  
  Matlab was used for computations, and Chebfun was utilized for the numerical computation of bessel functions and its derivatives. The integration of functions with explicit formulas was performed using the {\it integral} function in Matlab. The integration involving bessel functions was performed using the {\it besselj} function in Matlab. For the integration of the
  eigenvectors of $H= - \Delta+\varepsilon V_0$, the 2D trapezoid rule was used. We used a mesh size of 301 equal increments. Twenty-five eigenfunctions were computed.

     Next, we consider the problem of maximizing the functional $J_N^V(\chi_\omega)$, 
    \begin{align}
J_N^V(\chi_{\omega})=\inf_{1\leq j\leq N}\int_{\omega} \limits\phi_j^2(x)\,dx.
\end{align}     
     over all subsets satisfying $|\omega|= L|\Omega|$, for some $L \in (0,1)$, the argument of the maximum of which is called the optimal set. 
According to Propositions 3.9 and 4.5 of \cite{ptz}, for the $-\Delta$ operator, the maximum value of $J(\chi_\omega)=\inf\limits_{1\leq j\leq N}\int_{\omega} \limits\phi_{j0}^2(x)$ for radial subsets $\omega$ of the form $\omega= \{(r,\theta) \in [0,1] \times [0,2\pi] | \theta \in \omega_0\}$ and measure $L\pi$, is $L$. In the case when $L=0.5$, the supremum is reached for all subsets $\omega$ of the form $\omega= \{(r,\theta) \in [0,1] \times [0,2\pi] | \theta \in \omega_0\}$ of measure $\pi/2$, where $\omega_\theta$ is any measurable subset of $[0,2\pi]$ such that $\omega$ and its symmetric image are complementary in $[0,2\pi]$.
To gain a better understanding of the case of the $H= -\Delta+\varepsilon V_0$ operator, we used $L=0.5$ and tested radial subsets of measure $0.5\pi$ (or half the area of the total disk) satisfying Proposition 4.5 of \cite{ptz}. In particular, we note results for $\omega_0= \{[0,\pi/4] \cup [\pi/2,3\pi/4] \cup [\pi,5\pi/4] \cup [3\pi/2,7\pi/4] \}$,
and $N=25$. 

    The results are shown in Table \ref{unit disk} and Table \ref{unit disk2}; the $\delta$ and $\varepsilon$ variables were varied. The values in the table show that in all cases, the value of $J_N^V(\chi_{\omega})$ is very close to 0.5, which is the answer for the $-\Delta$ operator.

\begin{table}[!htb]
\begin{center}

\begin{tabular}{ | c | c | c | c | c | c | }
\hline
  $\varepsilon/\delta$  &       0.1.        &       0.2.           &         0.3.             &        0.4          &          0.475. \\
\hline
0.01       &      0.499999996  &  0.499999997    &  0.499999763   &  0.499995606    &   0.499999756  \\
\hline
0.05      &       0.5  &    0.499999988   &  0.499998816  & 0.499987946  & 0.499998898 \\
\hline
0.1        &        0.5     &   0.499999975  &   0.499997638  &  0.499999650  &  0.499998085 \\
\hline
0.5        &       �0.5    &   0.499999988  & 0.499998824 &   0.499988114   & 0.499998896 \\
\hline
1          &      0.499999999    &    0.499999997    &  0.499999764   &  �0.499995642  &  0.499999756 \\
\hline
\end{tabular}

\caption{Value of $J_N^V$ on unit disk with $V=1/r^2$}
\label{unit disk}
\end{center}
\end{table}

\begin{table}[!htb]
\begin{center}

\begin{tabular}{ | c | c | c | c | c | c | }
\hline
  $\varepsilon/\delta$  &       0.1.        &       0.2.           &         0.3.             &        0.4          &          0.475. \\
\hline
0.01       &       0.5   &   0.5     &   0.499999995    &     0.499999759      &       0.499998584   \\
\hline
0.05      &         0.5   &    0.5  &           0.499999975      &     0.499998825  & 0.499999896 \\
\hline
0.1        &        0.5    &     0.5    &   0.499999950  &  0.499997720  &     0.499999794  \\
\hline
0.5        &     0.499999999     &   0.499999999  &    0.499999748    &  0.499991449 &    0.499999063     \\
\hline
1          &      0.499999997    &    0.499999998  &     0.499999496   &   0.499990010   &   0.499998365 \\
\hline
\end{tabular}

\caption{Value of $J_N^V$ on unit disk with $V=r$}
\label{unit disk2}
\end{center}
\end{table}


\newpage

\section{Appendix: Convergence of Algorithms}
In order to provide an accurate numerical scheme, we also prove several theorems about $J^V(\chi_{\omega})$ and the problem of maximizing the functional. First, we prove convergence of the truncated version of $J^V(\chi_{\omega})$ for generic potentials:

\begin{thm}\label{main2}
Let
\begin{align}
J_N^V(\chi_{\omega})=\inf\limits_{1\leq j\leq N}\int\limits_{\omega}\phi_j^2(x)\,dx, \hspace{0.5cm} J_N^V(a)=\inf\limits_{1\leq j\leq N}\int\limits_{\Omega}a(x)\phi_j^2(x)\,dx
\end{align}
Then, the following statements hold:
\begin{enumerate}
\item For every measurable subset $\omega$ of $\Omega$, the sequence $(J_N^V(\chi_{\omega}))_{N\in\mathbb{N}^*}$ is non increasing and converges to $J^V(\chi_{\omega})$.
\item 
The following equality holds:
\begin{align}
\lim\limits_{N\rightarrow\infty}\max\limits_{a\in\overline{\mathcal{M}}_L}J^V_N(a)=\max\limits_{a\in\overline{\mathcal{M}}_L}J^V(a).
\end{align}
Moreover, whenever $(a^N)_{n\in\mathbb{N}^*}$ is a sequence of maximisers of $J_N^V$ in $\overline{\mathcal{M}}_L$, then up to a subsequence, this converges to a maximiser of $J$ in $\overline{\mathcal{M}}_L$ for the weak star topology of $L^{\infty}$.
\item Assume that $\Omega$ is a bounded analytic domain with boundary $\partial\Omega$. We have that $\forall N\in\mathbb{N}^*$, the problem $\max\limits_{\chi_{\omega} \in\overline{\mathcal{M}}_L}J^V_N(\chi_{\omega})$ has a unique solution $\chi_{\omega^N}$, where $\omega^N\in \mathcal{M}_L$. Moreover, the set $\omega^N$ is semi-analytic and has a finite number of connected components. 
\end{enumerate}
\end{thm}

We show that this relaxed problem allows for the determination of the observability constant under some assumptions on the flow. 
\begin{thm}\label{nogap}
Assume there exists a subsequence of the sequence of probability measures $\mu_j=\phi_j^2\,dx$, which converges vaguely to the measure $\frac{1}{|\Omega|}\,dx$ (Weak Quantum Ergodicity assumption with a potential). Then, the sequence of eigenfunctions $\phi_j$ is uniformly bounded in $L^{\infty}(\Omega)$ and 
\begin{align}
J^V(\chi_{\omega})=\sup\limits_{\omega \in \mathcal{M}_L}\inf\limits_{j\in \mathbb{N}^*}\int\limits_{\omega}\phi_j^2(x)\,dx=\sup\limits_{a\in \overline{\mathcal{M}}_L}\inf\limits_{j\in\mathbb{N}^*}\int\limits_{\Omega}a(x)\phi_j(x)^2\,dx=L.
\end{align}
\end{thm}
The assumptions of the above Theorem are sufficient but not necessary to derive such a no-gap statement between the original formulation of the problem and the relaxed formulation. It is known that when $\Omega$ is a two-dimensional disk and $V(x)\equiv 0$, the same statement holds true, even though WQUE (weak quantum unique ergodicity) is not satisfied \cite{ptz}. 

\begin{proof}[Proof of Theorem \ref{main2}]
To formulate the proof, we use the same steps as in \cite{ptz} in the proof of Theorem 4.9. These steps follow identically using the eigenfunctions $\phi_j$ of the $-\Delta+V$ operator instead of the eigenfunctions of the $-\Delta$ operator. We omit the steps here. 
\end{proof}

\begin{proof}[Proof of Theorem \ref{nogap}]
To conclude the proof, we use the same steps as in \cite{ptz} in the proof of Theorem 3.5. These steps follow identically using the eigenfunctions $\phi_j$ of the $-\Delta+V$ operator instead of the eigenfunctions of the $-\Delta$ operator. We omit the steps here. 
 \end{proof}

\begin{lem}\label{constant}
The convexified problem $\sup\limits_{a\in\overline{\mathcal{M}}_L}J(a)$ has at least one solution and 
\begin{align}
\label{quantity}
\sup\limits_{a\in\overline{\mathcal{M}}_L}\inf\limits_{j\in\mathbb{N}^*}\int\limits_{\Omega}a(x)\phi_j^2(x)\,dx=L.
\end{align}
The supremum is reached for the constant function $a(\cdot)=L$ on $\Omega$.
\end{lem}

\begin{proof}[Sketch of the proof of Lemma \ref{constant}]
The first statement follows from the fact that $J(a)$ is upper semicontinuous for the $L^{\infty}$ topology. In order to prove the second statement, we use the Ces\`{a}ro means of eigenfucntions. The constant function shows the fact that \eqref{quantity} is bounded below by $L$. We have
\begin{align}
\sup\limits_{a\in\overline{\mathcal{M}}_L}\inf\limits_{j\in\mathbb{N}^*}\int\limits_{\Omega}a(x)\phi_j^2(x)\,dx \leq \inf\limits_{j\in\mathbb{N}^*} \frac{1}{N} \sum_{j=1}^N \int\limits_{\Omega}a^{\star}(x)\phi_j^2(x)\,dx,
\end{align}
where $a^{\star}$ is a solution of the convexified problem. By using a similar argument as in the proof of Lemma 3.3 in \cite{ptz} regarding the uniform $|\Omega|^{-1}$ limit of the sequence $N^{-1}\sum_{j=1}^N \phi_j^2$ of Ces\`{a}ro means, one can show that \eqref{quantity} is bounded above by $L$. (The properties of Ces\'{a}ro means for the eigenfunctions are trivially satisfied by the spectral theorem).

\end{proof}

\newpage

\section*{Acknowledgments}

A.~W.~acknowledges support by EPSRC grant EP/L01937X/1.


\begin{thebibliography}{10}
\bibitem{AM} N. Anantharaman, F. Macia. \emph{Semiclassical measures for the Schrodinger equation on the torus}. J. Eur. Math. Soc. (JEMS) 16(6) 2014 pp 1253-1288. 

\bibitem{AFM} N. Anantharaman, C. Fermanian-Kammerer, F. Macia. \emph{Semiclassical completely integrable systems: dynamics and observability via two-microlocal Wigner measures.} Amer. J. Math., 137(3) (2015) 577-638.

\bibitem{AML} N. Anantharaman, M. Leautard, F. Macia. \emph{Wigner measures and observability for the Schrodinger equation on the disk} Invent. Math. 206(2) (2016) 485-599.

\bibitem{blr} C.~Bardos, G.~Lebeau and J.~Rauch, \emph{Sharp sufficient conditions for the observation, control, and stabilization of waves from the boundary}, SIAM J.~Control Optim.~30 (1992), 1024 - 1065.
\bibitem{BW} D.D.~ Bleecker and L. C.~ Wilson, \emph{Splitting the spectrum of a Riemannian manifold}, Siam J. Math. Anal. 11 (1980), no 5. 813-818. 
\bibitem{bu} D.~Bucur and G.~Buttazzo, \emph{Variation methods in shape optimization problems}, Progress in Nonlinear Differential Equations 65, Birkh\"{a}user Verlag, Basel (2005). 
\bibitem{burq} N.~Burq, \emph{Controlabilite exacte des ondes dans des ouverts peu reguliers}, Asymptot.~Anal.~14 (1997), 157 - 191.
\bibitem{burqgerard} N.~Burq and P.~Gerard, \emph{Condition necessaire et suffisante pour la controlabilite exacte des ondes}, C.~R.~Acad.~Sci.~Paris.~I Math. 325 (1997), no. 7, 749 - 752.
\bibitem{BZ1} N.~ Burq and M.~Zworski, \emph{Geometric control in the presence of a black box}, J. Amer. Math. Soc. 17 (2004) 443-471
\bibitem{BZ2} N.~Burq and M.~Zworski, \emph{Control for Schrodinger equations on tori}, Math. Res. Lett. 19 (2012) 309-324
\bibitem{conway} J.~Conway \emph{A course in Functional Analysis} 2nd ed. Graduate texts in mathematics. Springer-Verlag New York Berlin Heidelberg. (1990).
\bibitem{Davies} E.~B.~Davies, \emph{Spectral Theory and Differential Operators}, Cambridge University Press, Cambridge, 1989. 
\bibitem{gw} H.~Gimperlein, and A.~Waters \emph{A deterministic optimal design problem for the heat equation}. Siam Journal of Optimal Control (2017).  
\bibitem{Greenlee} W. M.~ Greenlee, \emph{Singular perturbation of simple eigenvalues}. Rocky Mountain J. Math. Vol 6, No. 4, 1976. 

\bibitem{G} E.~ Legendre, V. Guillemin, and R. Sena-Dias. \emph{Simple spectrum and Rayleigh quotients} Geometric and Spectral Analysis. ed. P.~ Albin, D.~ Jakobson, and F.~ Rochon. 
Contemp. Math. AMS. Rhode Island. (2014). 

\bibitem{Haraux} A. Haraux, \emph{Series lacunaires et control semi-interne des vibrations d'une plaque rectangulaire} J. Math. Pures Appl. 86 (1989), 457-465
\bibitem{hor} L.~H\"{o}rmander, \emph{The Analysis of Linear Partial Differential Operators},Grundlehren Math. Wiss. , 3-4, Springer, Berlin (1986), 274-275.
\bibitem{JW} J.~Ilmavirta and A.~Waters \emph{Recovery of Coefficients for the Acoustic Wave equation from Phaseless measurements} Comm. Math. Sci. 2018 (to appear). 
\bibitem{Jaffard} S.~ Jaffard, \emph{Control interne exact des vibrations d'une plaque rectangulaire} Portugal. Math 47 (1990) 423-429
\bibitem{Jin} L.~Jin \emph{Control for Schrodinger equation on Hyperbolic surfaces} \textrm{https://arxiv.org/pdf/1707.04990.pdf}.
\bibitem{kato} T.~ Kato, \emph{Perturbation Theory} Springer-Verlag Heidelberg Berlin, New York, 1980. 
\bibitem{Komoronik} V. Komoronik \emph{On the exact internal controllability of a Petrowsky system}, J. Math. Pure. Apple. 71 (1992) 33-342. 
\bibitem{kumar} S.~Kumar, and J.H.~Seinfeld, \emph{Optimal location of measurements for distributed parameter estimation}, IEEE Trans. Automat. Control ~23 (1978), 690 - 698.
\bibitem{camille} C.~Laurent \emph{Global controllability and stabilization for the nonlinear Schr\"odinger equation on an interval} ESAIM-COCV, 16(2): 356–379, 2010.
\bibitem{camille2} C.~ Laurent \emph{Global controllability and stabilization for the nonlinear Schr\"odinger equation on some compact manifolds of dimension 3} SIAM Journal on Mathematical Analysis, 42(2):785-832, 2010.

\bibitem{L} G. Lebeau. Controle de l'equation de Schrodinger. J. Math. Pures Appl. 71 (1992) 267-291. 

\bibitem{lr} G.~Lebeau, and L.~Robbiano, \emph{Controle exact de l'equation de la chaleur}, Comm.~Partial Differential Equations 20 (1995), 335 - 356.

\bibitem{Macia} F.~Macia, \emph{The Schrodinger flow on a compact manifold: High-frequency dynamics, and dispersion} Modern Aspects of the Theory of Partial Differential Equations, Oper. Theory Adv. Apple., 216, Springer, Basel 2011 275-289.

\bibitem{MR} F.~Macia, and G.~Riviere \emph{Concentration and Non-Concentration for the Schrodinger Evolution on Zoll Manifolds} Comm. Math. Phys. 345 (2016) 3, 1019-1054.

\bibitem{MR2} F.~Macia, and G.~Riviere \emph{Observability and quantum limits for the Schrodinger equation on the sphere} arXiv:1702.02066 (2017). 

\bibitem{morris} K.~Morris, \emph{Linear-quadratic optimal actuator location}. IEEE Trans. Automat. Control ~56 (2011), 113-124.
\bibitem{mu} A.~Munch and F.~Periago, \emph{Optimal distribution of the internal null control for 1D heat equation}, J. Diff. Equations 250 (2011), 95 - 111. 
\bibitem{ptz} Y.~Privat, E.~Trelat and E.~Zuazua, \emph{Optimal observability of the multidimensional wave and Schroedinger equations in quantum ergodic domains}, J.~Europ.~Math.~Soc.~18 (2016), 1043 - 1111.
\bibitem{opt} Y.~Privat, E.~Trelat and E.~Zuazua, \emph{Optimal observability of the one-dimensional wave equation}, J.~Fourier Anal.~Applications 19 (2013), no.~3, 514 - 544.
\bibitem{ptzheat} Y.~Privat, E.~Trelat and E.~Zuazua, \emph{Optimal shape and location of sensors for parabolic equations with random initial data}, Arch.~Rational Mech.~Anal.~216 (2015), 921 - 981.
\bibitem{bingyu} L.~Rosier, and B.~Zhang \emph{Local exact controllability and stabilizability of the nonlinear Schr\"odinger equation on a bounded interval} SIAM J. Control  Optim. 48 (2009), no. 2, 972--992.
\bibitem{bingyu2} L.~Rosier, and B.~Zhang \emph{Exact boundary controllability of the nonlinear Schr\"odinger equation}  J. Diff. Eqns.  246 (2009), no. 10, 4129--4153.
\bibitem{sigmund} O. Sigmund, J.S. Jensen, \emph{Systematic design of phononic band-gap materials and structures by topology optimization}, Roy. Soc. London Philos. Trans. Ser. A Math. Phys. Eng. Sci.~ 361 (2003), 1001-1019.
\bibitem{alex} A. ~Strohmaier \emph{Computation of Eigenvalues, Spectral Zeta Functions and Zeta-Determinants on Hyperbolic Surfaces}(2017) Contemporary Mathematics, vol. 700: Geometric and Computational Spectral Theory, pp. 177-206. Seminaire de Mathematiques Superieures (SMS 2015). 
\bibitem{uc} D.~Ucinski, M.~Patan, \emph{Sensor network design fo the estimation of spatially distributed processes}, Int. J. Appl. Math. Comput. Sci. ~20 (2010), 459-481.

\bibitem{Uh} K.~ Uhlenbeck, \emph{Generic properties of eigenfunctions} Amer. J. Math ~98 (1976), no 4, 1059-1078.
\bibitem{wal} M.~van de Wal, and B.~Jager, \emph{A review of methods for input/output selection}, Automatica ~37 (2001), 487-510.

\end{thebibliography}
\end{document}